\documentclass[10pt]{elsarticle}
\usepackage{amssymb}
\usepackage{amsthm}
\usepackage{enumerate}
\usepackage{amsmath}
\usepackage{graphicx}
\usepackage{tikz}
\usetikzlibrary{shapes.arrows, decorations.markings, arrows.meta}

\tikzset{
  tachado/.style={
    postaction=decorate,
    decoration={
      markings,
      mark=at position 0.5 with {
        \draw[-] (-2pt,-2pt) -- (2pt,2pt);
      }
    }
  }
}

\newtheorem{theorem}{Theorem}[section]
\newtheorem{lemma}[theorem]{Lemma}
\newtheorem{coro}[theorem]{Corollary}
\newtheorem{prop}[theorem]{Proposition}
\newtheorem{defi}[theorem]{Definition}
\newtheorem{note}[theorem]{Note}
\newtheorem{example}[theorem]{Example}
\newtheorem{question}[theorem]{Question}

\journal{a journal}

\begin{document}

\begin{frontmatter}

\title{Connectivity degrees of complements of closed sets in continua}

\author{Mauricio Chac\'on-Tirado}
\ead{mauricio.chacon@correo.buap.mx}
\author{C\'esar Piceno}
\ead{cesarpicman@gmail.com}

\address{Benem\'erita Universidad Aut\'onoma de Puebla, Facultad de Ciencias F\'isico Matem\'aticas, Avenida San Claudio y 18 Sur, Colonia San Manuel. Edificio FM1-101B, Ciudad Universitaria, C.P. 72570,
Puebla, Mexico.}

\begin{abstract}
In the literature, various types of points and meager sets whose complements are connected have been studied, such as colocally connected points, non-weak cut points/sets, non-block points/sets, shore points/sets, etc. We extend that study, in the following way: considering a continuum $X$ and a natural number $n$, we investigate sets $A \in 2^X$ meeting the criterion that  $X - A$ has at most $n$ components, and we introduce degrees of connectivity of the complement of \(A\). When \(n=1\) and \(A\) is meager or singleton, these new definitions are equivalent to the known definitions of non-cut points/sets.
\end{abstract}


\end{frontmatter}


\section{Introduction}
\label{introduction}
One of the main topics of interest in topology is being able to determine whether a space is connected or not, and,
when a space $X$ is connected, it is interesting to determine how ``strongly connected'' $X$ is. In the case of continua, various types of points and sets whose complements are connected have been studied, and some ``degree of connectivity'' of these complements has also been investigated. One of the most relevant works in this regard was published by R.L. Moore \cite{moore}, where the existence of non-cut points in all continua is demonstrated. Another important result concerning the degree of connectivity of the complement of a point in a continuum is the one obtained by Bing \cite{scc}, which states that for any point of a nondegenerate metrizable continuum, there is a proper continuumwise connected dense subset containing that point. Some of the articles that can be consulted on the topic are
\cite{noncutshore}, \cite{onblockers1}, \cite{non-cut}, \cite{blockers}, \cite {minc},  \cite{nall} and \cite{whyburn}.

If the space is not connected, we are also interested in knowing if it is composed of a finite number of components and how ``strongly disconnected'' is the space. In a continuum, it is of particular interest to study sets that cut the space, and if they do (or not), we are also interested in knowing the degree of connectivity of their complements. For example, in Section \ref{xhn(x)}, we show that in a locally connected continuum $X$, if $A \in 2^X$ and $X-A$ has a finite number of components, then each component is continuumwise connected.

Besides this introduction, this paper contains 4 more sections. In Section \ref{definitions}, we provide the definitions that we will use throughout the paper. In particular, we define degrees of connectivity, which we call $n$-$Q1$ to $n$-$Q7$, and $n$-$Qo$, being $n$-$Q1$ the stronger one and $n$-$Q7$ the weakest one. Something we wish to emphasize is the uniformity of the definitions for classifying the degree of connectivity of a space, which makes some results straightforward.

In Section \ref{generalproperties}, for each degree of connectivity we consider the hyperspace of closed sets whose complements have that degree of connectivity,  we explore the relationships between those hyperspaces, whose elements we call non-$n$-cut sets. We provide conditions to ensure when a complement that is $n$-$Q7$ implies also that  is $n$-$Q1$. We offer tools to discover new non-$n$-cut sets from others and ultimately examine the Borel classes of some hyperspaces of non-$n$-cut sets.

In Section \ref{continuousfunctions}, we delve into investigating the types of functions that preserve non-$n$-cut sets under their image or preimage.

Finally, in Section \ref{xhn(x)}, we investigate the relationships between a continuum $X$ and its hyperspaces of non-$n$-cut sets. Among other interesting results, we provide a characterization of the arc. We discover that for irreducible continua, certain hyperspaces of non-$n$-cut sets coincide. Additionally, we prove that if $X$ is aposyndetic with respect to $A$ and the complement of $A$ has at most $n$ components, then for every neighborhood $U$ containing $A$, there exists a neighborhood $V \subset U$ of $A$ such that the complement of $V$ has at most $n$ components.

\section{Definitions and notation}
\label{definitions}
In this paper all spaces are metric. The set \(\mathbb N\) represents the positive integers. Given a subset $A$ of a space $X$, the closure and the interior
of $A$ are denoted by $cl_X(A)$ and $int_X(A)$, respectively, and we omit the subindex when we feel there is no risk of confusion regarding our space.
A \textit{map} is a continuous function. A \textit{continuum} is a compact space with more than one point.
A continuum $X$ is \textit{aposyndetic at} $p$ \textit{with respect to} $A$, where $p \in X$ and $A \subset X$,
provided that there is a continuum $B\subset X - A$ such that $p \in int_X(B)$. A continuum $X$ is \textit{aposyndetic with respect to} $A$ if $X$ is aposyndetic at $p$ with respect to \(A\) for all $p \in X-A$.
A continuum $X$ is \textit{aposyndetic at} $p$ provided that $X$ is aposyndetic at $p$ with respect to each singleton $\{q\} \subset X - \{p\}$. A continuum $X$ is \textit{mutually aposyndetic} if for two distinct points $p,q \in X$, there exist two subcontinua $A$ and $B$ of $X$ such that $p \in \text{int}(A)$, $q \in \text{int}(B)$, and $A \cap B = \emptyset$. 

A compact metric space $X$ is \textit{indecomposable} provided that each subcontinuum of \(X\) has empty interior. A continuum $X$ is said to be \textit{irreducible about} $A \subset X$ provided that no proper subcontinuum of $X$ contains $A$. A continuum $X$ is said to be \textit{irreducible} provided that $X$ is irreducible about $\{p,q\}$ for some $p, q \in X$, in which case we say $X$ is \textit{irreducible between} $p$ \textit{and} $q$.
A space $Y$ is \textit{continuumwise connected} if any pair of points is contained in a continuum $X\subset Y$. Let $\mathcal S^1 = \{x \in \mathbb R^2 : ||x||=1\}$.

Given a non-empty space $X$ and \(n\in\mathbb N\), we consider the following \textit{hyperspaces} of $X$:

$$2^X= \{A \subset X : A \text{ is non-empty and compact} \},$$
$$M(X)= \{A \in 2^X : A \text{ has empty interior}\},$$
$$ C_n(X)= \{A \in 2^X : A \text{ has at most $n$ components}\},$$
$$D_0(X)= \{A \in 2^X : A \text{ has dimension }0\},$$
and
$$F_n(X)= \{A \in 2^X : A \text{ has at most $n$ elements}\}.$$ 
These hyperspaces are endowed with the Hausdorff metric.
We write $C(X)$ instead of  $C_1(X)$, the elements of \(C(X)\) are called \emph{subcontinua} of \(X\).

Clearly $F_n(X) \subset C_n(X) \subset 2^X$ and
$F_1(X)$ is homeomorphic to $X$. 

For a finite collection  $X_1,  \ldots , X_m$ of subsets of $X$, we define 
$\left\langle X_1,\ldots , X_m \right\rangle$ as the set
$\{A \in 2^X: A \subset X_1 \cup \ldots \cup
X_m \text{ and } A \cap X_i \neq \emptyset \text{ for each } i\in\{1, \ldots, m\}\}$.

It is known that if $X_1,  \ldots , X_m$ are closed subsets of $X$, then\\
$\left\langle X_1,  \ldots , X_m \right\rangle$ is closed in $2^X$
and that the collection of all subsets of the form $\left\langle U_1,  \ldots , U_m \right\rangle$,
where $U_1, \ldots , U_m$ are open subsets of $X$, is a base for the topology of $2^X$ (see \cite{cnofXII}).


The objective of the following definition is to introduce the degree of connectivity of a space.

\begin{defi}\label{principal2}
Given a non-empty space $X$ and $n\in\mathbb N$, we say that $X$ is:
\begin{enumerate}

\item n-Q1 \label{colocalconnect2} if there exists $B\in F_n(X)$, such that for each $x\in X$, there exists a continuum $D\subset X$ such that \(x\in int(D)\) and $B \cap D\neq\emptyset$;

\item n-Q2 \label{nwcn2} if there exists $B \in F_n(X)$ such that, for each $x\in X$, there exists a continuum $D\subset X$ such that $x\in D$ and $B \cap D \neq \emptyset$;

\item n-Q3  \label{n points2} if for every $x \in X$, there exists $B \in F_n(X)$ with $x \in B$, such that for every non-empty open set $U$ of \(X\), there exists a continuum $D\subset X$ such that \(B \cap D \neq \emptyset\neq D \cap U\);



\item n-Qo  \label{dnbopens2} if there exists \(B\in F_n(X)\) such that for every non-empty open set \(U\) of \(X\), there exists a continuum \(D\subset X\) such that \(B\cap int(D)\neq\emptyset\neq int(D)\cap U\);

\item n-Q4  \label{somen2} if there exists $B \in F_n(X)$, such that for every non-empty open set $U$ of \(X\), there exists a continuum $D\subset X$ such that $B \cap D \neq \emptyset\neq D\cap U$;

\item n-Q5 \label{nshore2} if for each finite family $\mathcal U$ of non-empty open sets contained in \(X\), there exists $D \in C_n(X)$ such that $D \cap U \neq \emptyset$, for every $U \in \mathcal U$;

\item n-Q6 \label{nscn2}if for each collection of \(n+1\) non-empty open sets \(U_1, \ldots, U_{n+1}\) of \(X\), there exists \(D \in C_n(X)\) such that \(D\cap U_i\neq\emptyset\) for every $i\in\{1\dots,n+1\}$.

\item n-Q7 \label{ncn2}if \(X\) has at most \(n\) components.


\end{enumerate}
    
\end{defi}

\begin{note}\label{relaciones}
Clearly, for each \(n\in\mathbb N\) and \(m\in\{o\}\cup\{1,\dots,6\}\), being \(n\)-$Qm$ implies being \((n+1)\)-\(Qm\), and being \(n\)-$Q(m+1)$, if \(m\in\{1,\dots,5\}\). Also, being \(n\)-$Q1$ implies being \(n\)-\(Qo\), being \(n\)-\(Qo\) implies being \(n\)-$Q4$, being \(n\)-\(Q4\) implies being \((n+1)\)-$Q3$, and being \(n\)-\(Q6\) implies having at most \(n\) components. In the following figure, what has been stated here is represented.
\end{note}

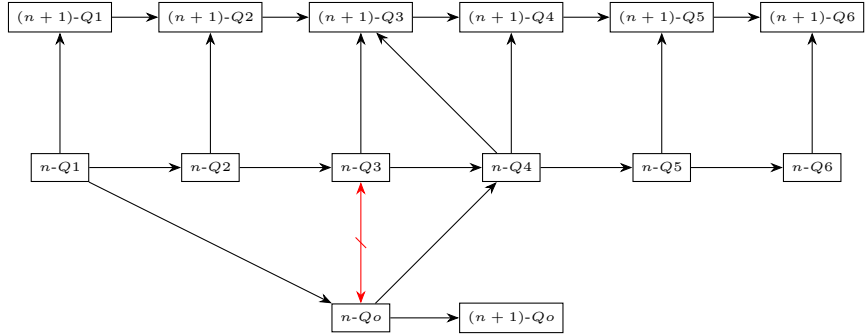
\begin{figure}[htb]
  \centering
  \small 
  \begin{tikzpicture}[node distance=2cm, >=Stealth, every node/.style={ draw, text centered, font=\tiny}]
    \node  (n1q1)                 {$(n+1)$-$Q1$};
    \node (n1q2) [right of=n1q1] {$(n+1)$-$Q2$};
    \node (n1q3) [right of=n1q2] {$(n+1)$-$Q3$};
    \node (n1q4) [right of=n1q3] {$(n+1)$-$Q4$};
    \node (n1q5) [right of=n1q4] {$(n+1)$-$Q5$};
    \node (n1q6) [right of=n1q5] {$(n+1)$-$Q6$};
       
    \node (nq1) [below of=n1q1]{$n$-$Q1$};
    \node (nq2) [right of=nq1] {$n$-$Q2$};
    \node (nq3) [right of=nq2] {$n$-$Q3$};
    \node (nq4) [right of=nq3] {$n$-$Q4$};
    \node (nq5) [right of=nq4] {$n$-$Q5$};
    \node (nq6) [right of=nq5] {$n$-$Q6$};
    \node (nqo) [below of=nq3] {$n$-$Qo$};
    \node (n1qo) [right of=nqo] {$(n+1)$-$Qo$};


    \draw[->] (n1q1) -- (n1q2);
    \draw[->] (n1q2) -- (n1q3);
    \draw[->] (n1q3) -- (n1q4);
    \draw[->] (n1q4) -- (n1q5);
    \draw[->] (n1q5) -- (n1q6);

    \draw[->] (nq1) -- (n1q1);
    \draw[->] (nq2) -- (n1q2);
    \draw[->] (nq3) -- (n1q3);
    \draw[->] (nq4) -- (n1q3);
    \draw[->] (nq4) -- (n1q4);
    \draw[->] (nq5) -- (n1q5);
    \draw[->] (nq6) -- (n1q6);
 
    \draw[->] (nq1) -- (nq2);
    \draw[->] (nq2) -- (nq3);
    \draw[->] (nq3) -- (nq4);
    \draw[->] (nq4) -- (nq5);
    \draw[->] (nq5) -- (nq6);
    \draw[->] (nq1) -- (nqo);
    \draw[->] (nqo) -- (nq4);

    \draw[->] (nqo) -- (n1qo);

    \draw[tachado, red, >=Stealth, <->] (nqo) -- (nq3);
  \end{tikzpicture}
  \caption{Relationships between degrees of connectivity.}
  \label{diagramarelacion}
\end{figure}
Notice that in Figure \ref{diagramarelacion}, it is indicated that there is no relationship between $n$-$Qo$ and $n$-$Q3$. Counterexamples to this fact are Example \ref{ejnbo} for $n$-$Qo \nRightarrow n$-$Q3$ and the Dyadic Solenoid for $n$-$Q3 \nRightarrow n$-$Qo$. On the other hand, we were unable to prove or deny $n$-$Q2 \Rightarrow n$-$Qo$. A question related to this, is Question \ref{nwcynbo}.

The following definitions are provided to classify sets based on the degree of connectivity of their
complements. Most of them are generalizations of those presented in \cite{noncutshore} and \cite{non-cut}.

\begin{defi}\label{principal}
Given a non-degenerate compact metric space $X$ and $n \in \mathbb N$, an element $A \in 2^X$ is said to be:
\begin{enumerate}

\item \label{colocalconnect} \textit{set colocal connectedness of degree \(n\)} of $X$ provided that \(A=X\) or $X-A$ is n-Q1;

\item \label{nwcn} \textit{not a weak cut set of degree \(n\)} of $X$ provided that \(A=X\) or $X-A$ is n-Q2;

\item \label{n points} \textit{nonblock set of degree \(n\)} of $X$ if \(A=X\) or $X-A$ is n-Q3;

\item \label{dnbopens} a set that \textit{does not block opens of degree \(n\)} of $X$ provided that \(A=X\) or  $X-A$ is n-Qo;

\item \label{somen} \textit{weak nonblock set of degree \(n\)} of $X$ provided that \(A=X\) or $X-A$ is n-Q4;

\item \label{nshore} a \textit{shore set of degree \(n\)} of $X$ provided that \(A=X\) or $X-A$ is n-Q5;

\item \label{nscn} \textit{not a strong center set of degree \(n\)} of \(X\) provided that \(A=X\) or $X-A$ is n-Q6.


\end{enumerate}
    
\end{defi}


We consider the following subspaces of $2^X$, 
these are called
\textit{hyperspaces of non-cut sets of degree \(n\)} of $X$:
\begin{enumerate}[I.]
\item ${CC}_n(X)= \{A \in 2^X :  A \text{ is a set of colocal connectedness  of degree \(n\) of \(X\)}\}$;
\item ${NWC}_n(X)= \{A \in 2^X: A \text{ is not a weak cut set of degree \(n\) of }X\}$;
\item ${NB}_n(X)= \{A \in 2^X: A \text{ is a nonblock set of degree \(n\) of }X\}$;
\item ${NBO}_n(X)= \{A \in 2^X: A \text{ does not block opens of degree \(n\) of }X\}$;
\item ${NB}_n^*(X)= \{A \in 2^X: A \text{ is a weak nonblock set of degree \(n\) of }X\}$;
\item $S_n(X)= \{A \in 2^X: A \text{ is a shore set of degree \(n\) of }X\}$;
\item ${NSC}_n(X)= \{A \in 2^X: A \text{ is not a strong center set of degree \(n\) of }X\}$;
\item ${NC}_n(X)= \{A \in 2^X:  X-A \text{ has at most \(n\) components}\}$.
\end{enumerate}

\begin{note}
For a continuum $X$, according to our definitions and the definitions P1, P2, P3, P4, P5  given in \cite{noncutshore}, $p$ is a P1 point if and only if $X-\{p\}$ is 1-Q1; $p$ is a P2 point if and only if $X-\{p\}$ is 1-Q2; $p$ is a P3 point if and only if $X-\{p\}$ is 1-Q4; $p$ is a P4 point if and only if $X-\{p\}$ is 1-Q5, and $p$ is a P5 point if and only if $X-\{p\}$ is 1-Q6.
\end{note}

\begin{note}\label{definicionesmagras}
The sets \(NWC(X)\), \( NB(F_1(X))\), \(NB^*(F_1(X))\), \(S(X)\), \(NC(X)\) defined in \cite{non-cut}, coincide with the sets $ {NWC}_1(X)\cap M(X)$, $ {NB}_1(X) \cap M(X)$, $ {NB}^*_1(X) \cap M(X)$, $S_1(X)\cap M(X)$ and $ {NC}_1(X)\cap M(X)$, respectively.
\end{note}


\section{General properties of non-$n$-cut sets}\label{generalproperties}

This section presents several key results in the context of hyperspaces of non-$n$-cut sets of continua. These results establish relationships between different families of non-$n$-cut sets, shedding light on their structural properties and interconnections.

The following theorem is immediate from Note \ref{relaciones}.
\begin{theorem}\label{contentions}
Given a continuum $X$ and $n \in \mathbb N$, the following conditions hold:
\begin{enumerate}
    \item If $H_n(X)$ represents an hyperspace of non-$n$-cut sets, then $H_n(X) \subset H_{n+1}(X)$;
    \item ${CC}_n(X) \subset  {NWC}_n(X)\subset  {NB}_n(X)\subset {NB}_n^*(X)\subset  S_n(X)\subset {NSC}_n(X)\subset  {NC}_n(X);$      \item $ {CC}_n(X) \subset {NBO}_n(X)\subset  {NB}_n^*(X)$; and
    \item $NB_n^*(X) \subset NB_{n+1}(X)$.
\end{enumerate}
\end{theorem}


The following example shows that the inclusion ${NBO}_1(X) \subset {NB}_1(X)$ is false in general. 

\begin{example}\label{ejnbo}
The \emph{harmonic fan} is the set \(([0,1]\times\{0\})\cup(\cup\{(x,\frac xn):x\in[0,1],n\in\mathbb N\})\), considered as a subspace of \(\mathbb R^2\).
Let $X$ be the harmonic fan and let $A=\{(\frac12,0)\}$. Then, $A\in {NBO}_1(X)$ and $A\notin {NB}_1(X)$.
\end{example}

From Example \ref{ejnbo} and 3 from Theorem \ref{contentions}, we know that there are continua $X$ such that $NWC_1(X) \neq NBO_1(X)$. What we do not know is the following:
\begin{question}\label{nwcynbo}
Does there exist a continuum $X$ such that $NWC_1(X) \not\subset NBO_1(X)$?    
\end{question}



\begin{note}
The examples ($P2 \setminus P1$), ($P4 \setminus P3$), ($P5 \setminus P4$) and ($P6 \setminus P5$) presented in \cite{noncutshore} satisfy that  $\bigcup_{n=1}^\infty {CC}_n(X) \subsetneq {NWC}_1(X)$, $\bigcup_{n=1}^\infty {NB^*}_n(X) \subsetneq S_1(X)$, $\bigcup_{n=1}^\infty {S}_n(X) \subsetneq NSC_1(X)$ and $\bigcup_{n=1}^\infty {NSC}_n(X) \subsetneq NC_1(X)$ respectively. While the examples a) and b) of Remark 3.3 of \cite{non-cut} satisfy that $\cup_{n=1}^{\infty} NWC_n(X) \subsetneq NB_1(X)$ and $NB_1(X) \subsetneq NB_1^*(X)$ respectively. Additionally, by Example \ref{ejnbo} and Theorem \ref{contentions}.3, we have that $CC_1(X) \subset NBO_1(X)$ can also be proper. Finally, in the case of the dyadic solenoid $X$, we have that $NBO_1(X) \subset NB_1^*(X)$ is also a proper inclusion.
\end{note}

\begin{defi}\label{colocalcerrado}
Let $X$ be a continuum and $A \in 2^X$. We say that $A$ is colocally connected of  degree closed if $A=X$ or there exists a set $B \in 2^X$ satisfying the following conditions: $B \cap A = \emptyset$ and for every $y \in X - A$, there exists a continuum $D \subset X - A$ such that $y \in \text{int}(D)$ and $D \cap B \neq \emptyset$.

We denote the set of all subsets of $X$ that are colocally connected  of degree closed by $CC_{2^X}$.
\end{defi}

\begin{theorem}\label{colocalclosed}
    For every continuum $X$, $CC_{2^X}= \bigcup_{n=1}^\infty CC_n(X)$.
\end{theorem}
\begin{proof}
The contention $\bigcup_{n=1}^\infty CC_n(X) \subset CC_{2^X}$ is clear. For the converse contention, 
    let $A \in CC_{2^X}$ and $B$ be a closed set satisfying the conditions of Definition \ref{colocalcerrado}. Then, for each $y \in X-A$, there exists a continuum $D_y \subset X-A$ such that $y \in int(D_y)$ and $D_y \cap B \neq \emptyset$. Hence, $\mathcal D= \{int(D_y): y \in X-A\}$ is an open cover of $B$. Given that $B \in 2^X$, there exists finite subcover $\{int(D_1), \ldots, int(D_n)\} \subset \mathcal D$ of $B$. For each $i \leq n$, choose $x_i \in D_i \cap B$. Let $B_n=\{x_1, \ldots, x_n\}$. Observe that for every $y \in X-A$, $D_y \subset X-A$ is such that $y \in int(D_y)$ and $D_y \cap B \neq \emptyset$. Hence, $D_y \cap D_i \neq \emptyset$ for some $i \leq n$. Therefore, $D=D_y \cup D_i$ is a continuum such that $D \subset X-A$, $y \in int(D)$ and $D \cap B_n \neq \emptyset$. Hence, $A \in CC_n(X)$. In conclusion, $CC_{2^X} \subset \bigcup_{n=1}^\infty CC_n(X)$.
\end{proof}

\begin{lemma}\label{lemaCCn}
    Let $X$ be a continuum and let $A \in CC_n(X)$, for some \(n\in\mathbb N\). If \(X-A\) is connected, then \(A\in CC_1(X)\).
\end{lemma}
\begin{proof}
    If \(A=X\), we are done. Assume \(A\neq X\).
    Let $B \in F_n(X)$ as in Definition \ref{principal2}.\ref{colocalconnect2} for $X-A$. For each \(b\in B\), let \(X_b=\{y\in X-A: \) there exists a continuum \(D\subset X-A \text{ such that }b\in D\text{ and }y\in int(D)\}\). Notice that each \(X_b\) is open and \(\{X_b:b\in B\}\) is a cover of \(X-A\). Assume \(b,c\in B\) satisfy \(X_b\cap X_c\neq\emptyset\), and let \(y\in X_b\cap X_c\) and \(z\in X_c\), thus there exist \(D,E,F\)  subcontinua of \(X-A\) such that \(y\in int(D), b\in D, y\in int(E), c\in E\), and \(z\in int(F), c\in F\), thus \(D\cup E\cup F\) is a subcontinuum of \(X-A\) containing \(z\) in its interior and containing \(b\), thus \(z\in X_b\), therefore \(X_c\subset X_b\), and analogously \(X_b\subset X_c\). In conclusion \(\{X_b:b\in B\}\) is a partition of \(X-A\), since \(X-A\) is connected, we have that \(X_b=X-A\) for some \(b\in B\). Thus the set \(\{b\}\) satisfies Definition \ref{principal}.\ref{colocalconnect} to show that \(A\in CC_1(X)\).
\end{proof}

Examples \ref{ejnbo} and \ref{dosenos} illustrate that we cannot replace the sets $CC_n(X)$ and $CC_1(X)$ in Lemma \ref{lemaCCn} with $NWC_n(X)$ and $NWC_1(X)$, $NB_n(X)$ and $NB_1(X)$, or $NBO_n(X)$ and $NBO_1(X)$, respectively.
\begin{theorem}\label{nmenorquem}
    Let \(H_n(X)\) represent a hyperspace of non-cut sets of degree \(n\) of a continuum \(X\). If \(CC_n(X)=H_n(X)\) for some \(n>1\), then \(CC_m(X)=H_m(X)\) for each \(m<n\).
\end{theorem}
\begin{proof}
    Let $m < n$. By Theorem \ref{contentions}, $CC_m(X) \subset H_m(X) \subset H_n(X)= CC_n(X)$. Let $A \in H_m(X)$, by Theorem \ref{contentions}, $A \in NC_m(X)$. Let $U_1 \ldots U_k$ be the components of $X-A$; given that $X-U_i \in NC_1(X)$ for each $i \leq k$, $X-U_i \in CC_1(X)$ (Lemma \ref{lemaCCn}). Hence, for each $i \leq k$, there exists $x_i \in U_i$ such that for each $y \in U_i$, there exists a continuum $D \subset U_i$ such that $y \in int(D)$ and $x_i \in D$. Therefore, for $B=\{x_1, \ldots , x_k\}$ and each $y \in X-A$, there exists a continuum $D$ such that $y \in int(D)$, $D \cap A = \emptyset$ and $D \cap B \neq \emptyset$. Given that $k \leq m$, $A \in CC_m(X)$. We conclude that $H_m(X)=CC_m(X)$.
\end{proof}

The following example shows that the converse of Theorem \ref{nmenorquem} is not true.

\begin{example}\label{dosenos}
    For the continuum $X=\{(x,\sin(\frac1x)):x\in [-1,0)\cup(0,1]\}\cup (\{0\}\times[-1,1])$, we have $CC_1(X)=NWC_1(X)=NB_1(X)=NB^*_1(X)=S_1(X)=NSC_1(X)$.
    However $CC_n(X) \neq NWC_n(X)$ and $CC_n(X) \neq NB_n(X)$, for each \(n\geq 3\), and  $CC_n(X) \neq NB^*_n(X)$, $CC_n(X) \neq S_n(X)$, $CC_n(X) \neq NSC_n(X)$ for each \(n\geq 2\).
\end{example}

Given the previous results, it is natural to ask the following.

\begin{question}\label{1implicam}
 For which spaces \(X\) and for which hyperspaces of non-cut sets of degree 1 $H_1(X)$, the following implication holds: if $CC_1(X)=H_1(X)$,  then  $CC_n(X)=H_n(X)$ for every $n \in \mathbb N$?
\end{question}
Partial answers to Question \ref{1implicam} are given in Theorem \ref{h1=cn1} and Corollary \ref{localyconnectn}.

\begin{theorem}\label{foreachcomp}
    Let $X$ be a continuum and $A \in 2^X$. If $A \in CC_n(X)$, then for each component $K$ of $X-A$, we have $A\in CC_1(A\cup K)$.
\end{theorem}
\begin{proof}
    If \(A=X\), we are done. Assume \(A\neq X\).
    Let $A \in CC_n(X)$, notice $A \in NC_n(X)$. Let $B \in F_n(X)$ as in Definition \ref{principal2}.\ref{colocalconnect2} for $X-A$ and let $K$ be a component of $X-A$. Let $y \in K$, thus there exists a continuum $D\subset X-A$ such that $B \cap D \neq \emptyset$, $y \in int(D)$. Since \(K\) is a component of \(X-A\) and \(y\in K\), then \(D\subset K\), therefore \(B\cap K\neq\emptyset\). Let \(B'=B\cap K\), thus \(B'\in F_n(X)\) satisfies Definition \ref{principal2}.\ref{colocalconnect2} for $K$ to show that \(A\in CC_n(A\cup K)\); by Lemma \ref{lemaCCn}, \(A\in CC_1(A\cup K)\).
\end{proof}

\begin{theorem}\label{laultimaref}
    Let \(H_n(X)\) represent a hyperspace of non-cut sets of degree \(n\) of a continuum \(X\). 
    If $\bigcup_{n=1}^\infty CC_n(X)=\bigcup_{n=1}^\infty H_n(X)$, then $CC_1(X)=H_1(X)$.
\end{theorem}
\begin{proof}
    Let $A \in H_1(X)$. Therefore, $A \in NC_1(X)$ and $A \in CC_m(X)$ for some $m \in \mathbb N$. Hence, Theorem \ref{foreachcomp} implies $A \in CC_1(X)$.
\end{proof}

\begin{theorem}\label{equivcomp}
    Let \(H_n(X)\) represent a hyperspace of non-cut sets of degree \(n\) of a continuum \(X\). Let $A\in 2^X$ be such that \(X-A\) has exactly  \(n\) components. Then, \(A\in  H_n(X)\) if and only if for each component \(K\) of \(X-A\), \( A\in  H_1(A\cup K)\).
\end{theorem}
\begin{proof}
Clearly, $A \in NC_n(X)$ and for each component $K$ of $X-A$, $A \in NC_1(A \cup K)$.

Let $K$ be a component of $X-A$.

Assume $A \in CC_n(X), NWC_n(X), NB_n^*(X)$ or $NBO_n(X)$. We prove that \(A\in CC_1(A\cup K), NWC_1(A \cup K), NB_1^*(A\cup K)\), or \(NBO_1(A\cup K)\), respectively. Let $B \in F_n(X-A)$ be a set that satisfies the respective definitions for $X-A$. Notice that $B$ must intersect each component of $X-A$, and the intersection of \(B\) with each component consists of exactly one point. Thus $K \cap B$ consists of only one point, and $K \cap B$ satisfies the respective definitions for $A$ to assert that $A \in CC_1(A \cup K), NWC_1(A \cup K), NB_1^*(A \cup K)$ or $NBO_1(A \cup K)$, respectively.

Assume $A \in NB_n(X)$. We prove that \(A\in NB_1(A\cup K)\). For each $x\in K$, we can find $B_x \in F_n(X-A)$ that satisfies the Definition \ref{principal2}.\ref{n points2} for $X-A$ and $x\in B_x$. Notice that $B_x$ must intersect each component of $X-A$, and the intersection of \(B_x\) with each component consists of exactly one point. Thus $B_x \cap K=\{x\}$ is the set that asserts that $A \in NB_1(A \cup K)$.

Assume $A \in S_n(X)$. We prove that \(A\in S_1(A\cup K)\). Let $\mathcal U$ be a finite family of open sets of \((A\cup K)-A=K\). Let $\mathcal V=\mathcal U \cup \{R: R $ is a component of $ X-A$ and \(R\neq K\}\). Given that $\mathcal V$ is a finite family of open sets contained in $X-A$, there exists $D \in C_n(X-A)$ such that $D$ intersects each element of the family $\mathcal V$. Since $D$ intersects each component of \(X-A\), $D$ has exactly $n$ components. Thus $D\cap K\in C_1(K)$ and $D \cap K$ intersects each element of $\mathcal U$. Therefore, $A \in S_1(A \cup K)$.

Assume $A \in NSC_n(X)$. Let $U,V$ be two open sets such that $U \cup V \subset K$. Let $\mathcal V=\{U, V\} \cup \{R: R $ is a component of $ X-A$ and \(R\neq K\}\). Given that $\mathcal V$ is a family of $n+1$ open sets contained in $X-A$, there exists $D \in C_n(X-A)$ such that $D$ intersects each element of the family $\mathcal V$. Since $D$ intersects each component of \(X-A\), $D$ has exactly $n$ components. Thus $D\cap K\in C_1(K)$ and $D \cap K$ intersects $U$ and $V$. Therefore, $A \in NSC_1(A \cup K)$.

Let $K_1, \ldots K_n$ be the components of $X-A$. Assume that for each $K_i$, $A \in CC_1(A \cup K_i)$, so \(K_i\) is a $1$-$Q1$ space.  For each $i\in \{1,\dots,n\}$, choose $x_i \in K_i$ such that $\{x_i\}$ witnesses that $K_i$ is a $1$-$Q1$ space. Then, $B=\{x_1, \ldots x_n\}$ witnesses $\bigcup_{i=1}^n K_i$ is a $n$-$Q1$ space, so $A \in CC_n(X)$. Analogously, if $A \in NWC_1(A \cup K_i)$ for each $i\in\{1,\dots n\}$, or $A\in NB_1^*(A \cup K_i)$ for each $i\in\{1,\dots n\}$, or $A\in NBO_1(A \cup K_i)$ for each $i\in\{1,\dots n\}$, then $A \in NWC_n(X),$ or $A\in NB_n^*(X)$, or $A\in NBO_n(X)$, respectively.

Assume that for each $i\in \{1,\dots,n\}$, $A \in NB_1(A \cup K_i)$, so \(K_i\) is a 1-$Q3$ space. For each $i \in \{1,\dots,n\}$, choose $x_i \in K_i$. For $x \in X-A$, let \(j\) be such that $x \in K_j$. Notice that $B=\{x\}\cup(\{x_1, \ldots x_n\}-\{x_j\})$ witnesses $\bigcup_{i=1}^n K_i$ is a $n$-$Q3$ space, so \(A\in NB_n(X)\).

Assume that for each $i\in \{1,\dots,n\}$, $A \in S_1(A \cup K_i)$, so \(K_i\) is a 1-$Q5$ space. Let $\mathcal U$ be a finite family of non-empty open sets contained in $\bigcup_{i=1}^n K_i$. For each \(i\in\{1,\dots,n\}\), let \(\mathcal U_i=\{K_i\}\cup \{U\cap K_i:U\in \mathcal U\text{ and }U\cap K_i\neq\emptyset\}\). Since \(K_i\) is a $1$-$Q5$ space, there exists $D_i \in C(K_i)$ such that $D_i \cap U \neq \emptyset$ for every $U \in \mathcal U_i$. Hence, $D= \cup_{i=1}^n D_i \in C_n(\bigcup_{i=1}^n K_i)$ satisfies $D \cap U \neq \emptyset$ for each $U \in \mathcal U$. Therefore \(\bigcup_{i=1}^n K_i\) is a $n$-$Q5$ space, so $A \in S_n(X)$.


Assume that for each $i\in \{1,\dots,n\}$, $A \in NSC_1(A \cup K_i)$, so \(K_i\) is a 1-$Q6$ space. Let $\mathcal{U} = \{U_1, \ldots, U_{n+1}\}$ be a family of $n+1$ non-empty open sets contained in $\bigcup_{i=1}^n K_i$. Let $K$ be a component of $X-A$ such that $K \cap U_i \neq \emptyset \neq K \cap U_j$ for some $i \neq j $.
Since $K$ is a $1$-$Q6$ space, there exists $D_K \in C(X)$ such that $D_K \cap U_i \neq \emptyset \neq \emptyset \neq D_K \cap U_j$. Now, for each $m \in\{1,\dots,n+1\}$, choose $x_m \in U_m$.
Hence, $D= D_K \cup \{x_m: m \in\{1,\dots,n+1\}-\{i,j\}\} \in C_n(X)$ satisfies $D \cap U \neq \emptyset$ for each $U \in \mathcal U$. Therefore \(\bigcup_{i=1}^n K_i\) is a $n$-$Q6$ space, so $A \in NSC_n(X)$.
\end{proof}

In Example \ref{dosenos}, if $A=\{(-1,0)\}$, the set $X-A$ has $1$ component. However, although $A$ belongs to $NWC_3(X), NB_3(X), NB^*_2(X), S_2(X)$ and $NSC_2(X)$, it does not belong to $NWC_1(X),$ $NB_1(X),$ $NB^*_1(X),$ $S_1(X),$ or $NSC_1(X)$. This implies that the result in Theorem \ref{equivcomp} is false if we remove the condition that $X-A$ has exactly $n$ components.

\begin{theorem}\label{h1=cn1}
    If $H_1(X)$ is a hyperspace of non-cut sets such that \(H_1(X)=NC_1(X)\), then \(H_n(X)=NC_n(X)\) for each \(n\in\mathbb N\).
\end{theorem}
\begin{proof}
    By Theorem \ref{contentions}, $H_n(X) \subset NC_n(X)$. Let $A \in NC_n(X)$. If \(A=X\), then \(A\in H_n(X)\). Suppose that $X-A$ has exactly $m$ components, for some \(m\in \{1,\dots,n\}\). Observe that for each component $K$ of $X-A$, $X-K \in NC_1(X)=H_1(X)$. Since \(X-(X-K)=(A\cup K)-A\), we have $A \in H_1(A \cup K)$. By Theorem \ref{equivcomp}, $A \in H_m(X) \subset H_n(X)$. Therefore $H_n(X)= NC_n(X)$.
\end{proof}


By Theorems \ref{colocalclosed}, \ref{laultimaref} and \ref{h1=cn1}, we obtain the following result.

\begin{coro}\label{ccnncn}
The following conditions are equivalent:
\begin{itemize}
    \item \(CC_1(X)=NC_1(X)\);
    \item \(CC_n(X)=NC_n(X)\) for some \(n\in \mathbb N\);
    \item \(CC_n(X)=NC_n(X)\) for each \(n\in \mathbb N\);
    \item \(\bigcup_{n=1}^\infty CC_n(X)=\bigcup_{n=1}^\infty NC_n(X)\);
    \item $CC_{2^X}=\bigcup_{n=1}^\infty NC_n(X)$.
\end{itemize}
\end{coro}

Recognizing or obtaining new non-cut sets of degree $n$ from those already known is of great interest; Theorems \ref{intaba} and \ref{component} address this.

\begin{theorem}\label{intaba}
    Let $X$ be a continuum and $n \in \mathbb N$. Let $A \in 2^X$ and let $C \in 2^X$ be such that $int(A) \subset C \subset A$.
    \begin{enumerate}
        \item If $A \in NBO_n(X)$, then $C \in NBO_n(X)$;
        \item if $A \in NB_n^*(X)$, then $C \in NB_n^*(X)$;
        \item if $A \in S_n(X)$, then $C \in S_n(X)$;
        \item if $A \in NSC_n(X)$, then $C \in NSC_n(X)$;
        \item if $A \in NC_n(X)$, then $C \in NC_n(X)$; and
    \end{enumerate}
\end{theorem}
\begin{proof}
\begin{enumerate}
    \item Let $A \in NBO_n(X)$ and let $B \in F_n(X-A)$ witnessing that $X-A$ if a $n$-$Qo$ space. Let $V$ be a non-empty open set of \(X-C\). Given that  $int(A) \subset C \subset A$, $V_A=V-A$ is a non-empty open set of $X-A$. Therefore, there exists a continuum $D \subset X-A$ such that $int(D) \cap V_A \neq \emptyset \neq int(D) \cap B$, which implies that $int(D) \cap V \neq \emptyset\neq int(D) \cap B$. Hence, \(B\) witnesses that \(X-C\) is a $n$-$Qo$ space, so $C \in NBO_n(X)$.

    \item Let $A \in NB_n^*(X)$, and let $B \in F_n(X-A)$  witnessing that \(X-A\) is a $n$-$Q4$ space. Given that $int(A) \subset C \subset A$, for every non-empty open set $V$ of $X-C$, $V_A= V-A \neq \emptyset$ is an open set of \(X-A\). Therefore, there exists a continuum $D \subset X-A \subset X-C$ such that $D \cap V_A \neq \emptyset\neq D\cap B$, which implies that $D \cap V \neq \emptyset$. Hence, Hence, \(B\) witnesses that \(X-C\) is a $n$-$Q4$ space, so $C \in NB_n^{*}(X)$.

    \item Let $A \in S_n(X)$ and let $\mathcal U$ be a finite family of non-empty open sets of \(X-C\). Given that $int(A) \subset C \subset A$, the family $\mathcal V=\{U-A: U \in \mathcal U\}$ is a finite family of non-empty open sets of \(X-A\). Therefore, there exists $D \in C_n(X-A)$ such that $D \cap V \neq \emptyset$ for each $V \in \mathcal V$, which implies that $D \cap U \neq \emptyset$ for each $U \in \mathcal U$. Hence, \(X-C\) is a $n$-$Q5$ space, so $C \in S_n(X)$.

    \item Let $A \in NSC_n(X)$ and let $U_1, \ldots, U_{n+1}$ be a collection of $n+1$ non-empty open sets of \(X-C\). For each $i \in\{1,\dots, n+1\}$, let  $V_i=U_i-A$. We have $V_1,\ldots , V_{n+1}$ is a collection of $n+1$ non-empty open sets of \(X-A\). Therefore, there exists $D \in C_n(X-A)$ such that $D \cap V_i \neq\emptyset$ for each $i \in\{1,\dots, n+1\}$. Hence, \(X-C\) is a $n$-$Q6$ space, so $C \in NSC_n(X)$.

    \item Notice that \(X-A\subset X-C\subset X-int(A)=cl(X-A)\). Hence, $X-C$ has at most the same number of components as $X-A$.
\end{enumerate}
\end{proof}

Example \ref{ultimoejemplo} shows that Theorem \ref{intaba} cannot be extended to the sets $CC_n(X)$, $NWC_n(X)$, and $NB_n(X)$. 

\begin{example}\label{ultimoejemplo}
   Let $X$ be the circle of pseudo-arcs and let $f:X \to \mathcal S^1$ be the quotient map from $X$ onto the circle described on \cite{circlepseudoarcs}. Then $f$ is an onto, monotone and open map. Hence, for every $x \in \mathcal S^1$, $f^{-1}(x) \in CC_1(X)$ (see Proposition \ref{mono} and Proposition \ref{monopen}), and no proper subset of $f^{-1}(x)$ is an element of $CC_1(X), NWC_1(X)$ neither $NB_1(X)$. 
\end{example}

The following lemma gives us a characterization of the elements in $NBO_n(X)$.
\begin{prop}\label{lemanbo}
     Let $X$ be a continuum and $A \in 2^X$. Then, $A \in NBO_n(X)$ if and only if for each non-empty finite family $\mathcal U$ of non-empty open sets contained in $X-A$, there exists $D \in C_n(X-A)$ such that $int(D) \cap U \neq \emptyset$ for all $U \in \mathcal U$.
\end{prop}
\begin{proof}
    Let $A \in NBO_n(X)$ and let $B \in F_n(X-A)$ witnessing that \(X-A\) is a $n$-$Qo$ space. Let $\mathcal U=\{U_1, \ldots U_m\}$ be a non-empty finite family of non-empty open sets contained in \(X-A\). Then, for each $i \in \{1,\dots,m\}$, there exists $D_i \in C(X-A)$, such that $int(D_i) \cap U_i \neq \emptyset\neq int(D_i)\cap B$. Therefore, $D =\cup_{i=1}^m D_i \in C_n(X-A)$ satisfies  $int(D) \cap U \neq \emptyset$ for each $U \in \mathcal U$.

    Now suppose that $A \in 2^X$ satisfies that for each non-empty finite family $\mathcal U$ of non-empty open sets of \(X-A\), there exists $D \in C_n(X-A)$ such that $int(D) \cap U \neq \emptyset$ for each $U \in \mathcal U$. If \(A=X\), we have \(A\in NBO_n(X)\). 
    
    Assume \(A\neq X\). For \(D\in C_n(X-A)\), define \(\alpha(D)=\bigcup\{K\in C(X-A):K\cap D\neq\emptyset\}\).

    \textbf{Claim}: There exists \(D\in C_n(X-A)\) such that \(\alpha(D)\) is dense in \(X-A\) and each component of \(D\) has non-empty interior.

    \textit{Proof of Claim}:

    Let $\mathcal F=\{\{K_1,K_2,\dots ,K_m\}: m\in\mathbb N, $ for each $i,j\in\{1,\dots,m\}$, $K_i\in C(X-A)$, $int(K_i)\neq\emptyset$, and if $i\neq j, \alpha(K_i)\cap \alpha(K_j)=\emptyset\}$.
    
    By the properties of $A$, if \(\{K_1,\dots,K_m\}\in \mathcal F\), then \(m\leq n\). Let \(M\) be the maximum number of elements of the members of \(\mathcal F\) and let \(\{K_1,\dots,K_M\}\in\mathcal F\).

    Notice that \(\bigcup_{i=1}^M \alpha(K_i)\) is dense in \(X-A\). Let \(D=\bigcup_{i=1}^M K_i\), notice that \(D\in C_n(X-A)\) and \(\alpha(D)\) is dense in \(X-A\). The claim is proved.

    Now, take $D$ as in the claim and let \(B\in F_n(X)\) containing exactly one point in the interior of each component of \(D\). If $U$ is any non-empty open set of $X-A$, then there exist a continuum $F$ such that $int(F)\cap U\neq \emptyset$ and a continuum $K \subset X-A$ such that $K \cap D \neq \emptyset \neq K \cap (U \cap int(F))$. 
    
    Let \(E\) be a component of \(D\) that intersects \(K\). Therefore,  $E\cup K \cup F \in C(X-A)$, \(B\cap int (E\cup K\cup F)\neq\emptyset\) and $int(E\cup K \cup F) \cap U \neq \emptyset$. Thus $B$ witnesses that $X-A$ is a $n$-$Qo$ space, so $A\in NBO_n(X)$.
\end{proof}

The following lemma gives us a characterization of the elements in $CC_n(X)$.

\begin{lemma}\label{lemacc}
    Let $X$ be a continuum, $A \in 2^X$ and \(n\in\mathbb N\). Then, $A \in CC_n(X)-\{X\}$ if and only if for each open set $U$ with $A \subset U$, there exists an open set $V$ such that $A \subset V \subset U$ and $X-V \in C_n(X)$. 
\end{lemma}
\begin{proof}
    Let $A \in CC_n(X)-\{X\}$ and let $U$ be an open set such that $A \subset U$. Let $B \in F_n(X-A)$ witnessing that \(X-A\) is a $n$-$Q1$ space. For each $y \in X-U$, let $D_y$ be a continuum such that $y \in int(D_y)$, $D_y \cap B \neq \emptyset$ and $D_y \subset X-A$. Since $X-U$ is compact and $\{int(D_y): y \in X-U\}$ is an open cover of $X-U$, there exists a finite subcover $\{int(D_1), \ldots, int(D_k)\}$ of $X-U$. Let $V=(X-\cup_{i=1}^kD_i) \subset U$, observe that \(V\) is an open set, \(A\subset V\subset U\) and $X-V=\cup_{i=1}^kD_i \in C_n(X)$.
    
    Now suppose that $A \in 2^X$ is such that for each open set $U$ such that $A \subset U$, there exists an open set $V$ such that $A \subset V \subset U$ and $X-V \in C_n(X)$. Observe that $X-A\neq\emptyset$ must have at most $n$ components. Choose $B \in F_n(X)$ such that \(B\) intersects each component of \(X-A\). Let $y \in X-A$, let \(V_y\subset X-A\) be a closed neighborhood of \(y\) and let \(U= X-(V_y\cup B)\). Since $A \subset U $, there exists $V$ open such that \(A\subset V\subset U\) and $X-V \in C_n(X)$, which implies that the component $D$ of $X-V \subset X-A$ containing $y$ is a continuum such that $B\cap D \neq \emptyset$ and \(y\in int(D)\).
\end{proof}

\begin{theorem}\label{component}
    Let $X$ be a continuum and $n \in \mathbb N$. Let $A, C \in 2^X$ such that $C$ is union of some components of \(A\).
    \begin{enumerate}
        \item If $A \in CC_n(X)$, then $C \in CC_n(X)$;
        \item if $A \in NWC_n(X)$, then $C \in NWC_n(X)$; and
        \item if $A \in NC_n(X)$, then $C \in NC_n(X)$.
    \end{enumerate}
\end{theorem}

\begin{proof}
    If \(A=X\), the result is trivial. Assume \(A\neq X\).
    \begin{enumerate}
    \item Let \(A\in CC_n(X)\) and $B \in F_n(X-A)$ witnessing that $X-A$ is a $n$-$Q1$ space.  Let $x \in X-C$. If $x \in X-A$, there exists a continuum $G$ containing $x$ in its interior such that $G \cap B \neq \emptyset$ and $B \subset X-A$. Now, assume $x \in A-C$, and let $D$ be the component of \(x\) in $A-C$. Since \(A\) is compact, there exist two open sets $U$ and $V$ such that $C \subset U$, $D \subset V$, $U \cap V = \emptyset$, $A \subset U \cup V$; moreover, by Lemma \ref{lemacc}, we may assume that $E=X-(U \cup V)$ has at most $n$ components. Therefore, $E\cup V=X-U\in C_n(X)$ (Theorem 5.6 from \cite{nadler}) and contains $x$ in its interior. Since the component of $E \cup V$ containing $x$ has a non-empty interior in $X-A$, there exists a continuum $G$ containing $x$ in its interior such that $G \cap B \neq \emptyset$ and $B \subset X-C$. Hence, $C \in CC_n(X)$.

    \item Let \(A\in NWC_n(X)\) and $B \in F_n(X-A)$ witnessing that $X-A$ is an $n$-$Q2$ space. Let $x \in X-C$. If $x \in X-A$, there exists a continuum $F$  containing $x$ such that $F\subset X-A$, and $F \cap B \neq \emptyset$. Now, assume $x \in A-C$. Let $D$ be the component of $A$ containing $x$. Since \(A\) is compact, there exist two open sets $U$ and $V$ such that $C \subset U$, $D \subset V$, $U \cap V = \emptyset$, $A \subset U \cup V$. Thus, there exists a continuum $G$ such that $D \subsetneq G \subset U$ (Corollary 5.5 of \cite{nadler}). Choose $r \in G-A$. Since $r\in X-A$, there exists a continuum $F$ containing $r$ such that $F \subset X-A$ and $F \cap B \neq \emptyset$. Therefore, $F \cup G$ is a continuum containing $x$ such that $F\subset X-C$, and $(F \cup G) \cap B \neq \emptyset$. In conclusion, $B$ witnesses that $X-C$ is an $n$-$Q2$ space, thus $C \in NWC_n(X)$.

    \item Let \(A\in NC_n(X)\). Let $x \in X-C$. If $x \in X-A$, the component of \(x\) in \(X-A\) is contained in the component of \(x\) in \(X-C\). Now, assume $x \in A-C$. Let $D$ be the component of $A$ containing $x$. Since \(A\) is compact, there exist two open sets $U$ and $V$ such that $C \subset U$, $D \subset V$, $U \cap V = \emptyset$, $A \subset U \cup V$. Thus, there exists a continuum $G$ such that $D \subsetneq G \subset U$ (Corollary 5.5 of \cite{nadler}). Choose $r \in G-A$. Since $r\in X-A$, the component of \(r\) in \(X-C\) contains \(G\) and contains the component of \(r\) in \(X-A\), and the component of \(x\) in \(X-C\) is the same as the component of \(r\) in \(X-C\). Thus, each component of \(X-C\) contains some component of \(X-A\), so \(C\in NC_n(X)\).
    \end{enumerate}
\end{proof}

The following examples show that we can not generalize the previous theorem to the hyperspaces ${NB}_n(X)$, $ {NBO}_n(X)$, $ {NB}_n^*(X)$, $ {S}_n(X)$ and $ {NSC}_n(X)$.

\begin{example}\label{compactificcioonNBO}
   Let $Y$ be the dyadic solenoid, let $S \subset Y$ be an arc and let $h:I\times \{0\} \to S$ be an homeomorphism. Define $X=Y \cup_h (I \times I)$, $A=\{p\} \cup (I \times I)$ where $p$ is a point of \(X\) not in the composant of $I \times I$, and let \(C=\{p\}\). Observe that $A\in NB_1(X)$ and $C \notin NB_1(X)$. 
\end{example}

\begin{example}\label{igualpenultimoejemplo}
   Let $Y$ be a compactification of the ray with remainder $\mathcal S^1$, let $S \subset \mathcal S^1$ be an arc and $h:I\times \{0\} \to S$ an homeomorphism. Define $X=Y \cup_h (I \times I)$, $A=\{p\} \cup (I \times I)$ where $p \in \mathcal S^1-S$, and \(C=\{p\}\). Observe that $A\in NBO_1(X)$, \(C\) is a component of \(A\) and $C\notin NSC_1(X)$. 
\end{example}

In Proposition 2.4 of \cite{noncutshore}, the authors studied what is the Borel type class of the sets $CC_1(X) \cap F_1(X)$, $S_1(X) \cap F_1(X)$, $NSC_1(X) \cap F_1(X)$, and $NC_1(X) \cap F_1(X)$. We extend the analysis for some hyperspaces of non-cut sets of degree $n$.

\begin{prop}
Let $X$ be a continuum and $n \in \mathbb N$. The following is true.
\begin{itemize}
\item [(i)] $CC_n(X)$ is of type $G_\delta$,
\item [(ii)] $NBO_n(X)$ is of type $G_\delta$,
\item [(iii)] $S_n(X)$ is of type $G_\delta$, and
\item [(iv)] $NSC_n(X)$ is of type $G_\delta$.
\end{itemize}
\end{prop}
\begin{proof}
\begin{itemize}
    \item[(i)] For each \(k\in\mathbb N\), define $CC_n^k$ as the union of all the sets $\left\langle U_1, \ldots , U_m \right\rangle\subset 2^X$ with $m \in \mathbb N$, $U_i$ non-empty and open in \(X\), $diam(U_i)<\frac{1}{k}$ for each \(i\), such that there exists $B \in F_n(X)$ satisfying that for every $z \in X-\bigcup_{i=1}^m U_i$, there exists a continuum $D \subset X-\bigcup_{i=1}^m U_i$ such that $z \in {int}(D)$ and $B \cap D \neq \emptyset$. It holds that $CC_n(X) = \bigcap_{k=1}^\infty CC_n^k$.
    

    \item[(ii)] Let $\mathcal B$ be a countable base with $\emptyset\notin\mathcal B$. For each $\mathcal U \subset \mathcal B$ finite, define the set $NBO_{\mathcal U}$ as follows:
    $$NBO_{\mathcal U} = \bigcup \{\langle X-K \rangle : K \in C_n(X), \forall U \in \mathcal U, int(K) \cap U \neq \emptyset  \} \cup (\bigcup_{U \in \mathcal U}\langle X, U \rangle).$$ Let $\mathcal C=\{\mathcal U : \mathcal U \text{ is a non-empty finite subset of } \mathcal B\}$.  It holds that $NBO_n(X) = \bigcap_{\mathcal U \in \mathcal C} NBO_{\mathcal U}$.
    

    \item[(iii)] Let \(\mathcal B\) and \(\mathcal C\) as in (ii). For each $\mathcal U \subset \mathcal B$ finite, define the set $S_{\mathcal U}$ as follows:
    $$S_{\mathcal U} = \bigcup \{\langle X-K \rangle : K \in C_n(X), \forall U \in \mathcal U, K \cap U \neq \emptyset  \} \cup (\bigcup_{U \in \mathcal U}\langle X, U \rangle).$$
    
    It holds that
    $S_n(X) = \bigcap_{\mathcal U \in \mathcal C} S_{\mathcal U}$.

    \item[(iv)] Let \(\mathcal B\) and \(S_\mathcal U\) as in (ii). 
    It holds that $NSC_n(X)=\bigcap \{ S_{\mathcal U}: \mathcal U \subset \mathcal B $ non-empty with at most $n+1 \text{ elements} \}$.
\end{itemize}
\end{proof}

In Example 2.5 from \cite{noncutshore}, the authors show a continuum $X$ where $NWC_1(X) \cap F_1(X)$ is not Borel, consequently, $NWC_1(X)$ itself is not Borel. In (ii), they show that the sets $CC_1(X) \cap F_1(X)$, $S_1(X) \cap F_1(X)$, $NSC_1(X) \cap F_1(X)$, and $NC_1(X) \cap F_1(X)$ do not necesarilly fall into the category of $F_\sigma$, implying the same for $CC_1(X)$, $S_1(X)$, $NSC_1(X)$, and $NC_1(X)$. Finally, in (iii), the authors furnish an example where $NC_1(X) \cap F_1(X)$ lacks the $G_\delta$ property, hence $NC_1(X)$ is not $G_\delta$.


\begin{theorem}
    If $X$ is a compact metric space, then $M(X)$ is a $G_\delta$ set in $2^X$.
\end{theorem}
\begin{proof}
    Let $\{U_n: n \in \mathbb N\}$ be a countable base of $X$. For each $n \in \mathbb N$, we define $$\mathcal U_n=\{A \in 2^X: U_n \subset A\}.$$ It is clear that each $\mathcal U_n $ is closed and $2^X - M(X) = \cup_{n} \mathcal U_n$ which is an $F_\sigma$ set. Hence, $M(X)$ is a $G_\delta$ set. 
\end{proof}

As a consequence of the following corollary, we can specify the Borel type class of some sets mentioned in Note \ref{definicionesmagras}.

\begin{coro}
Let $X$ be a continuum and $n \in \mathbb N$. The following is true.
\begin{itemize}
\item [(i)] $CC_n(X) \cap M(X)$ is of type $G_\delta$,
\item [(ii)] $NBO_n(X) \cap M(X)$ is of type $G_\delta$,
\item [(iii)] $S_n(X) \cap M(X)$ is of type $G_\delta$, and
\item [(iv)] $NSC_n(X) \cap M(X)$ is of type $G_\delta$.
\end{itemize}
\end{coro}

\section{Properties of non-$n$-cut sets preserved by continuous functions}\label{continuousfunctions}

The results presented herein highlight how different types of mappings, such as onto mappings, open mappings, and monotone mappings, preserve some properties of non-$n$-cut sets.

We start with a lemma. 
\begin{lemma}\label{casiApertura}
Let $f:X \to Y$ be an onto mapping between continua. If $y\in Y$, and \(D\in 2^X\) are such that \(f^{-1}(y)\subset int_X(D)\), then \(y\in int_Y(f(D))\).
\end{lemma}
\begin{proof}Let \(y\in Y\), \(D\in 2^X\) and assume \(f^{-1}(y)\subset int_X(D)\).
Notice that \(y\in Y-f(X-int_X(D))\subset f(D)\) and \(Y-f(X-int_X(D))\) is open in \(Y\). Thus \(y\in int_Y(f(D))\).
\end{proof}

\begin{prop}\label{ontomapimplication}
Let $f:X \to Y$ be an onto mapping between continua, let $A \in 2^Y$ and let \(n\in\mathbb N\). The following statements hold:
\begin{enumerate}
    \item If $f^{-1}(A) \in CC_n(X)$, then $A \in CC_n(Y)$;
    \item if $f^{-1}(A) \in NWC_n(X)$, then $A \in NWC_n(Y)$;
    \item if $f^{-1}(A) \in NB_n(X)$, then $A \in NB_n(Y)$;
    \item if $f^{-1}(A) \in NB_n^*(X)$, then $A \in NB_n^*(Y)$;
    \item if $f^{-1}(A) \in S_n(X)$, then $A \in S_n(Y)$;
    \item if $f^{-1}(A) \in NSC_n(X)$, then $A \in NSC_n(Y)$; and
    \item if $f^{-1}(A) \in NC_n(X)$, then $A \in NC_n(Y)$.
\end{enumerate}
\end{prop}

\begin{proof}
Observe that the statements are true if \(A=Y\). Suppose that \(A\neq Y\).
\begin{enumerate}
   \item Assume that $f^{-1}(A) \in CC_n(X)$. Let $B \in F_n(X-f^{-1}(A))$ witnessing that $X-f^{-1}(A)$ is a $n$-$Q1$ space. Let $y \in Y-A$. For each \(x\in f^{-1}(y)\), there exists a continuum \(D_{x}\) contained in \(X-f^{-1}(A)\) such that \(x\in int_X(D_{x})\) and \(B \cap D_{x} \neq \emptyset\). By compactness of \(f^{-1}(y)\), we can find a finite number of elements \(x_1,\dots,x_n\in f^{-1}(y)\) such that \(f^{-1}(y)\subset \bigcup_{i=1}^m D_{x_i}\). Notice that \(\bigcup_{i=1}^m D_{x_i}\subset X-f^{-1}(A)\) is in \(C_n(X)\) and contains \(f^{-1}(y)\) in its interior. By Lemma \ref{casiApertura}, \(f(\bigcup_{i=1}^m D_{x_i})\subset Y-A\) is a subcontinuum of \(Y\) such that \(f(B) \cap f(\bigcup_{i=1}^n D_{a_i}) \neq \emptyset\) and has \(y\) in its interior, so \(f(B)\) witnesses that \(Y-A\) is a $n$-$Q1$ space, therefore \(A\in CC_n(Y)\).

   \item Assume that $f^{-1}(A) \in NWC_n(X)$. Let $B \in F_n(X)$ witnessing that $X-f^{-1}(A)$ is a $n$-$Q2$ space, let $y\in Y - A$ and $x \in f^{-1}(y)$. Since $X-f^{-1}(A)$ is $n$-$Q2$, there exists a continuum $D \subset X$ such that $D \cap f^{-1}(A) = \emptyset$, $w \in D$ and $D \cap B \neq \emptyset$. Therefore $f(D) \subset Y-A$ is a continuum, $x \in f(D)$ and $f(B) \cap f(D) \neq \emptyset$, so \(f(B)\) witnesses that \(Y-A\) is a $n$-$Q2$ space, therefore \(A\in NWC_n(Y)\).
 
   \item Assume that $f^{-1}(A) \in NB_n(X)$. Let $y \in Y-A$ and $x \in f^{-1}(y)$. Let $B \in F_n(X)$ witnessing that $X-f^{-1}(A)$ is a $n$-$Q3$ space and \(x\in B\). Let \(U\subset Y-A\) be and non-empty open set. Notice that \(f^{-1}(U)\subset X-f^{-1}(A)\) is a non-empty open set of \(X\); since \(X-f^{-1}(A)\) is $n$-$Q3$, there exists a continuum \(D\subset X-f^{-1}(A)\) such that $D \cap B \neq \emptyset$ and \(D\cap f^{-1}(U)\neq \emptyset\), thus \(f(D)\subset Y-A\) is a continuum such that $f(D) \cap f(B) \neq \emptyset$, \(f(D)\cap U\neq \emptyset\) and $y \in f(B)$, so \(f(B)\) witnesses that \(Y-A\) is a $n$-$Q3$ space, therefore \(A\in NB_n(Y)\).
    
   \item Assume that $f^{-1}(A) \in NB_n^*(X)$. Let $B \in F_n(X)$ witnessing that $X-f^{-1}(A)$ is a $n$-$Q4$ space. Let \(U\subset Y-A\) be and non-empty open set. Notice that \(f^{-1}(U)\subset X-f^{-1}(A)\) is a non-empty open set of \(X\); since \(X-f^{-1}(A)\) is $n$-$Q4$, there exists a continuum \(D\subset X-f^{-1}(A)\) such that $D \cap B \neq \emptyset$ and \(D\cap f^{-1}(U)\neq \emptyset\), thus \(f(D)\subset Y-A\) is a continuum such that $f(D) \cap f(B) \neq \emptyset$ and \(f(D)\cap U\neq \emptyset\), so \(f(B)\) witnesses that \(Y-A\) is a $n$-$Q4$ space, therefore \(A\in NB_n^*(Y)\).
  
   \item Assume that $f^{-1}(A) \in S_n(X)$. Let \(U_1,\dots, U_m\) be a finite number of non-empty open sets contained in \(Y-A\). Since \(f^{-1}(U_1),\dots,f^{-1}(U_m)\) is a finite number of non-empty open sets contained in \(X-f^{-1}(A)\) and $X-f^{-1}(A)$ is a $n$-$Q5$ space, there exists an element $D \in C_n(X)$ such that \(D\subset X-f^{-1}(A)\), \(D\cap f^{-1}(U_i)\neq\emptyset\) for each \(i\in\{1,\dots,m\}\), hence \(f(D)\subset Y-A\) is an element of $C_n(Y)$ such that \(f(D)\cap U_i\neq\emptyset\) for each \(i\in\{1,\dots,m\}\). So \(Y-A\) is a $n$-$Q5$ space, therefore \(A\in S_n(Y)\).
   
   \item Assume that $f^{-1}(A) \in NSC_n(X)$. Let \(U_1,\dots, U_{n+1}\) be non-empty open sets contained in \(Y-A\). Since \(f^{-1}(U_1),\dots,f^{-1}(U_{n+1})\) are non-empty open sets contained in \(X-f^{-1}(A)\) and $X-f^{-1}(A)$ is a $n$-$Q6$ space, there exists an element $D \in C_n(X)$ such that \(D\subset X-f^{-1}(A)\), \(D\cap f^{-1}(U_i)\neq\emptyset\) for each \(i\in\{1,\dots,n+1\}\), hence \(f(D)\subset Y-A\) is an element of $C_n(Y)$ such that \(f(D)\cap U_i\neq\emptyset\) for each \(i\in\{1,\dots,n+1\}\). So \(Y-A\) is a $n$-$Q6$ space, therefore \(A\in NSC_n(Y)\).

   \item Assume that $f^{-1}(A) \in NC_n(X)$. Since \(Y-A=f(X-f^{-1}(A))\) and \(X-f^{-1}(A)\) has at most \(n\) components, then \(Y-A\) has at most \(n\) components. Therefore \(A\in NC_n(X)\).
\end{enumerate}
\end{proof}

In the following example, we show that it is not possible to extend Proposition \ref{ontomapimplication} to the hyperspace $NBO_n(X)$. A weaker result is presented in Proposition \ref{weaker}.

\begin{example}
Let $Y$ be the Knaster buckethandle continuum, $p \in Y$ be the endpoint of $Y$, and let $\alpha:[0,1)\rightarrow K$ be an onto injective mapping such that $\alpha(0)=p$, where \(K\) is the composant of \(Y\) containing \(p\). Define $X=(Y\times\{0\})\cup\{(\alpha(t),1-t): t\in[0,1)\}\subset Y\times[0,1]$ and let $f:X\to Y$ be the projection onto $Y$. Then $f$ is an onto mapping. Note that for $z \in Y - K$, $f^{-1}(z)=\{(z,0)\} \in  {NBO}_1(X)$, and $\{z\} \notin \bigcup_{n=1}^\infty {NBO}_n(Y)$.
\end{example}

\begin{prop}\label{weaker}
Let $f:X \to Y$ be an onto and open mapping between continua, let $A \in 2^Y$ and let \(n\in\mathbb N\). If $f^{-1}(A) \in NBO_n(X)$, then $A \in NBO_n(Y)$.
\end{prop}
\begin{proof}
    Assume that $f^{-1}(A) \in NBO_n(X)$. If $A=Y$, then $A\in NBO_n(X)$. Assume \(A\neq Y\). Let $B \in F_n(X)$ witnessing that $X-f^{-1}(A)$ is a $n$-$Qo$ space.
    Let $U$ be an open set of $Y-A$. Then, there exists $D \in C(X)$ such that $D\subset X-f^{-1}(A)$ and $B\cap D \neq \emptyset \neq int(D) \cap f^{-1}(U)$. Hence, $f(D) \subset Y-A$ and $f(B)\cap f(D) \neq \emptyset $, and since $f$ is open, $int(f(D)) \cap U \neq \emptyset$, which implies that $Y-A$ is a $n$-$Qo$ space, therefore \(A\in NBO_n(Y)\).
\end{proof}

\begin{prop}\label{mono}
Let $f:X \to Y$ be an onto monotone mapping between continua, let $A \in 2^Y$ and let \(n\in\mathbb N\). The following statements hold:
\begin{enumerate}
    \item If $A \in CC_n(Y)$, then $f^{-1}(A) \in CC_n(X)$;
    \item if $A \in NWC_n(Y)$, then $f^{-1}(A) \in NWC_n(X)$; and
    \item if $A \in NC_n(Y)$, then $f^{-1}(A) \in NC_n(X)$.
\end{enumerate}
\end{prop}

\begin{proof}
Observe that the statements are true if \(A=Y\). Suppose that \(A\neq Y\).
\begin{enumerate}
    \item Assume $A \in CC_n(Y)$. Let $B \in F_n(Y-A)$ witnessing that $Y-A$ is a $n$-$Q1$ space and $B' \in F_n(X- f^{-1}(A))$ such that $f(B')=B$. Let \(x \in X-f^{-1}(A)\). Choose $D \in C(Y)$ such that \(D\subset Y-A\), \(B \cap D \neq \emptyset\) and \(f(x)\in int_Y(D)\). Hence \(f^{-1}(D)\subset X-f^{-1}(A)\) is a continuum with \(B' \cap f^{-1}(D) \neq \emptyset\) and \(x\in int_X(f^{-1}(D))\). So \(f^{-1}(A)\in CC_n(X)\).

     \item Assume $A \in CC_n(Y)$. Let $B \in F_n(Y-A)$ witnessing that $Y-A$ is a $n$-$Q2$ space and $B' \in F_n(X- f^{-1}(A))$ such that $f(B')=B$. Let \(x \in X-f^{-1}(A)\). By the properties of $B$, there exists a continuum \(D\subset Y-A\) such that \(f(x)\in D\) and $B \cap D \neq \emptyset$. Hence, \(f^{-1}(D)\subset X-f^{-1}(A)\) is a continuum with \(x\in f^{-1}(D)\) and $f^{-1}(D) \cap B' \neq \emptyset$. So \(f^{-1}(A)\in NWC_n(X)\).

    \item Let $A \in NC_n(Y)$. Observe that $f^{-1}(Y-A)=X-f^{-1}(A)$ has at most $n$ components, therefore $f^{-1}(A) \in NC_n(X)$.
\end{enumerate}
\end{proof}

The following example shows that Theorem \ref{mono} does not hold for the hyperspaces $NB_n(Y), NB_n^*(Y), S_n(Y)$ and $NSC_n(Y)$. We do not know if Proposition \ref{mono} holds for $NBO_n(X)$.
\begin{example}\label{penultimoejemplo}
   Let $Y$ be the dyadic solenoid, let $S \subset Y$ be an arc and let $h:I\times \{0\} \to S$ be a homeomorphism. Take $X=Y \cup_h (I \times I)$  and let $f:X \to Y$ be defined as $f(x)=x$, if $x \in Y$, and $f(x)=h(x_1,0)$ if $x=(x_1,x_2) \in I \times I$. Observe that $f:X \to Y$ is an onto, monotone mapping. However, for $A=\{ h(\frac 12,0)\} \in NB_1(Y)$ and $f^{-1}(A) \notin NSC_1(X)$.
\end{example}

However, if we add to the conditions of Theorem \ref{mono} that the map \(f:X\to Y\) is open, Theorem \ref{mono} holds for the remaining hyperspaces of non-cut sets.

\begin{prop}\label{monopen}
Let $f:X \to Y$ be an onto, open and monotone mapping between continua, let $A \in 2^Y$ and let \(n\in\mathbb N\). The following statements hold:
\begin{enumerate}
    \item If $A \in NB_n(Y)$, then $f^{-1}(A) \in NB_n(X)$;
    \item if $A \in NBO_n(Y)$, then $f^{-1}_n(A) \in NBO_n(X)$;
    \item if $A \in NB_n^*(Y)$, then $f^{-1}(A) \in NB_n^*(X)$;
    \item if $A \in S_n(Y)$, then $f^{-1}(A) \in S_n(X)$; and
    \item if $A \in NSC_n(Y)$, then $f^{-1}(A) \in NSC_n(X)$.
\end{enumerate}
\end{prop}
\begin{proof}
Observe that the statements are true if \(A=Y\). Suppose that \(A\neq Y\).
\begin{enumerate}
    \item Assume $A \in NB_n(Y)$. Let \(x\in X-f^{-1}(A)\). Since \(Y-A\) is a $n$-$Q3$ space, let $B\in F_n(Y-A)$ satisfying that \(f(x)\in B\) and for each  non-empty open set \(U\subset Y-A\), there exists a continuum \(D\subset Y-A\) such that \(D\cap B\neq\emptyset\neq D\cap U\).

    Let \(B'\in F_n(X-f^{-1}(A))\) be such that \(x\in B'\) and \(f(B')=B\).       
    If $V \subset X-f^{-1}(A)$ is a non-empty open set of $X$, then $f(V)$ is open in \(Y-A\), so there exists a continuum $D$ such that $B \cap D \neq \emptyset\neq D \cap f(V)$. Hence, $f^{-1}(D)\subset X-f^{-1}(A)$ is a continuum and $f^{-1}(D)\cap V \neq \emptyset\neq f^{-1}(D)\cap B'$. So, \(X-f^{-1}(A)\) is a $n$-$Q3$ space and $f^{-1}(A) \in NB_n(X)$.

    \item Assume $A \in NBO_n(Y)$. Let $B \in F_n(Y-A)$ witnessing that $Y-A$ is a $n$-$Qo$ space. Let \(B'\in F_n(X-f^{-1}(A))\) be such that \(f(B')=B\). Consider a non-empty open set $U\subset X-f^{-1}(A)$. Since $f(U)\subset Y-A$ is non-empty and open, there exists a continuum $D\subset Y-A$ such that $B \cap int (D) \neq \emptyset\neq int(D) \cap f(U)$. Hence $f^{-1}(D)\subset X-f^{-1}(A)$ is a continuum and $B'\cap int(f^{-1}(D))\neq\emptyset\neq int(f^{-1}(D))\cap U $. So \(X-f^{-1}(A)\) is a $n$-$Qo$ space and $f^{-1}(A) \in NBO_n(X)$.

    \item Assume $A \in NB_n^*(Y)$. Let $B \in F_n(Y-A)$ witnessing that $Y-A$ is a $n$-$Q4$ space. Let \(B'\in F_n(X-f^{-1}(A))\) be such that \(f(B')=B\). Consider a non-empty open set $U\subset X-f^{-1}(A)$. Since $f(U)\subset Y-A$ is non-empty and open, there exists a continuum $D\subset Y-A$ such that $B \cap D\neq \emptyset\neq D\cap f(U)$. Hence $f^{-1}(D)\subset X-f^{-1}(A)$ is a continuum and $B'\cap f^{-1}(D)\neq\emptyset\neq f^{-1}(D)\cap U $. So \(X-f^{-1}(A)\) is a $n$-$Q4$ space and $f^{-1}(A) \in NB_n^*(X)$.

   \item Assume $A \in S_n(Y)$. Let \(U_1,\dots,U_m\) be a finite number of non-empty open sets contained in \(X-f^{-1}(A)\). Since \(f(U_1),\dots,f(U_m)\) is a finite number of non-empty open sets contained in $Y-A$ and $A \in S_n(Y)$, there exists an element $D \in C_n(Y-A)$ such that \(D\cap f(U_i)\neq\emptyset\) for each \(i\in\{1,\dots,m\}\). Hence $f^{-1}(D) \in C_n(X-f^{-1}(A))$, satisfies \(f^{-1}(D)\cap U_i\neq\emptyset\) for each \(i\in\{1,\dots,m\}\). So, \(f^{-1}(A)\in S_n(X)\).

    \item Assume $A \in NSC_n(Y)$. Let \(U_1,\dots,U_{n+1}\) be a finite number of non-empty open sets contained in \(X-f^{-1}(A)\). Since \(f(U_1),\dots,f(U_{n+1})\) is a collection of $n+1$ non-empty open sets contained in $Y-A$ and $A \in NSC_n(Y)$, there exists an element $D \in C_n(Y)$ such that $D\subset Y-A$ and \(D\cap f(U_i)\neq\emptyset\) for each \(i\in\{1,\dots,n+1\}\). Hence $f^{-1}(D) \in C_n(X)$, \(f^{-1}(D)\subset X-f^{-1}(A)\) and \(f^{-1}(D)\cap U_i\neq\emptyset\) for each \(i\in\{1,\dots,n+1\}\). So, \(f^{-1}(A)\in NSC_n(X)\).
\end{enumerate}
\end{proof}

\section{Relations between $X$ and the hyperspaces of non-cut sets of degree $n$ of $X$}\label{xhn(x)}

 The first result we present provides an initial insight into the relationship between hyperspaces of non-$n$-cut sets, when the original space exhibits certain specific characteristics. We will explore properties of the original space that may lead to the coincidence of some of its hyperspace of non-$n$-cut sets. On the other hand, while the fact that a set is a non-$n$-cut set may not be particularly relevant by itself, studying the hyperspace of these sets can lead to interesting conclusions.

\begin{theorem}\label{localyconnect}
If X is a locally connected continuum, then
${NC}_1(X) =   {CC}_1(X)$.
\end{theorem}
\begin{proof}
We only have to prove that \( {NC}_1(X)\subset  {CC}_1(X)\). Let $A \in  {NC}_1(X)$ and let \(x,y\in X-A\). Let $U_x$ and \(U_y\) be open connected neighborhoods of $x$ and \(y\), respectively, such that $cl(U_x) \cap A= \emptyset=cl(U_y)\cap A$. Since $X-A$ is an open connected set, $X-A$ is arcwise connected (Theorem 8.26 of \cite{nadler}), so let \(Y\) be an arc in \(X-A\) joining \(x\) to \(y\). Observe that $D= Y\cup cl(U_x)\cup cl(U_y)$ is a continuum avoiding \(A\) with \(x,y\in int(D)\). In conclusion, $A \in  {CC}_1(X)$.
\end{proof}

\begin{coro}\label{localyconnectn}
If X is a locally connected continuum, then
$  {CC}_n(X) =   {NC}_n(X)$ for each \(n\in\mathbb N\).
\end{coro}
\begin{proof}
See Corollary \ref{ccnncn}.
\end{proof}

The following example shows that Theorem \ref{localyconnect} does not hold if we replace the locally connected continuum condition with the mutually aposyndetic continuum condition or the Kelley continuum condition.

\begin{example}
Let $X$ the suspension over the Cantor set, so \(X\) is mutually aposyndetic and Kelley. Let $p$ and $q$ the vertices of $X$, and let $A$ be a set consisting of two points in the same arc-component of $X-\{p,q\}$.
Then, $A \in   NC_1(X),$ $A\notin NB_1(X)$ and $A \notin CC_1(X)$. Notice that \(F_1(X)=NC_1(X)\cap F_1(X)=CC_1(X)\cap F_1(X)\).
\end{example}

We do not know a non-locally connected continuum \(X\) where $NC_1(X)=NWC_1(X)$. Hence, we find the following question interesting:

\begin{question}\label{pregunta}
If ${CC}_1(X)=  {NC}_1(X)$ or ${NWC}_1(X)=  {NC}_1(X)$, is $X$ locally connected?
\end{question}

\begin{note}\label{indecompo}
When $X$ is an indecomposable continuum, we have $ {CC}_n(X)= {NWC}_n(X)= {NBO}_n(X)=\{X\}$ and $C_n(X) \subset  {NB}_n^*(X)$, for each \(n\in\mathbb N\).
\end{note}


With the following example, we show that $\{X\}={CC}_n(X)={NWC}_n(X)$ $={NBO}_n(X)$ does not imply that the space is indecomposable. 

\begin{example}
    For a space \(X=A\cup B\), where \(A\) and \(B\) are two indecomposable continua and \(A\cap B\) consists only of one point, we have \(CC_n(X)=NWC_n(X)=NBO_n(X)=\{X\}\), for each \(n\in\mathbb N\).
\end{example}


As mentioned earlier, the condition of $X$ being locally connected implies that $NC_n(X)=CC_n(X)$ for each \(n\in\mathbb N\). As part of the research related to Question \ref{pregunta}, for a continuum $X$ and a set $A \in 2^X$, we are exploring certain conditions under which $A \in NC_n(X)$ implies $A \in CC_n(X)$. With that goal, we generalize the definition of semi-local connectivity given by Whyburn in \cite{whyburn} from points to sets, and we define a continuum $X$ to be \textit{semi-locally connected at a set} $A$ provided that if $U$ is an open subset of $X$ containing $A$, there is an open subset $V$ of $X$ lying in $U$ and containing $A$ such that $X - V$ has a finite number of components.

\begin{prop}\label{0dim}\label{equivalencia}
Let $X$ be a continuum. If $A \in  \bigcup_{i=1}^\infty CC_i(X)$, then $X$ is semi-locally connected at $A$.
\end{prop}
\begin{proof}
It follows from Lemma \ref{lemacc}.
\end{proof}

While it may seem intuitive that the converse of Theorem \ref{equivalencia} holds, this is not the case. In the example \ref{following}, the vertex is a point of semi-local connectivity but does not belong to $\bigcup_{i=1}^\infty CC_i(F_\omega)$.

\begin{example}\label{following}
Define $F_\omega=\bigcup_{n=1}^\infty \{(x,\frac xn):x\in[0,\frac1n]\}$, considered as a subspace of \(\mathbb R^2\).
Let $A=\{(0,0)\}$. Then, $F_\omega$ is semi-locally connected at $A$ but $A\notin \bigcup_{i=1}^\infty CC_i(F_\omega)$.
\end{example}

The following theorem is based on results (6.1) and (6.21) of \cite{whyburn}.

\begin{theorem}\label{semiloc}
Let $X$ be a continuum and $A \in 2^X$. If $X$ is semi-locally connected at $A$, then each component of $X-A$ is a continuumwise connected open set. 
\end{theorem}

\begin{proof}
    Let $D$ be a component of $X-A$. Let $\{U_n\}_{n=1}^\infty$ be a sequence of open sets such that for each \(n\), $A\subset U_n $, $cl(U_{n+1}) \subset U_n$, \(A=\bigcap_{n=1}^\infty U_n\) and $X-U_n$ has a finite number of components. For each \(a\in D\), let \(C_a=\bigcup \{K: K\) is the component of \(X-U_n\) containing \(a\), for some \(n\in\mathbb N\}\). Notice that for each \(a,b\in D\), \(a\in C_a\), \(C_a\) is continuumwise connected and \(C_a\cap C_b\neq \emptyset\) implies \(C_a=C_b\).
    
    Now we prove that \(C_a\) is open for each \(a\in D\). If $x \in C_a$, then $x$ is an element the component of \(X-U_n\) containing \(a\), for some \(n\in\mathbb N\). Therefore \(x\notin cl(U_{n+1})\), so \(x\in int(K)\), where \(K\) is the component of \(X-U_{n+1}\). Hence $x \in int(C_a)$.
    
    Since, $D-C_a=\cup \{C_b: b \in D- C_a\}$ is an open set and \(D\) is connected, $C_a$ must be equal to $D$.
\end{proof}

The following theorem is useful to understand better the characteristics of the sets $A$, for which $X$ is semi-locally connected at $A$.

\begin{theorem}\label{apossemi}
Let $X$ be a continuum and $A \in 2^X$. Then, \(X\) is aposyndetic with respect to $A$ if and only if it is semi-locally connected at $A$.
\end{theorem}
\begin{proof}
    Assume that $X$ is aposyndetic with respect to $A$ and let $U$ be an open set such that $A \subset U$. Then, for each $p \in X-A$, there exists a subcontinuum $C_p$ of $X$ such that $p \in int(C_p)$ and $C_p\cap A = \emptyset$. Hence, $\{int(C_p): p \in X-A\}$ is an open cover of $X-U$. Therefore, there exists a finite set $\{p_1, \ldots , p_n\} \subset X-A$ such that $X-U \subset \bigcup _{i=1}^nC_{p_i}$. Notice that $V = X-\bigcup _{i=1}^nC_{p_i} \subset U$ and \(X-V\) has at most \(n\) components. In conclusion, $X$ is semi-locally connected at $A$.
    
    Now, suppose that $X$ is semi-locally connected at $A$ and let $x \in X-A$.
    Let $U$ be an open set such that $A \subset U \subset cl(U) \subset X-\{x\}$. Hence, there exists an open set $V$ such that $A \subset V \subset U$ and $X-V$ has a finite number of components. Therefore, the component of $X-V$ containing $x$ is a continuum with $x$ in its interior. We conclude that $X$ is aposyndetic with respect to $A$. 
\end{proof}

\begin{coro}\label{ncccc}
    Let $X$ be a continuum and let $A\in 2^X$ such that $X$ is semi-locally connected at $A$. If $A \in NC_1(X)$, then $A \in CC_1(X)$.
\end{coro}
\begin{proof}
    Let $x,y$ two different points in $X-A$, then by Theorem \ref{semiloc}, there exists a continuum $K$ such that $\{x,y\} \subset K$ and $K \cap A= \emptyset$. By Theorem \ref{apossemi}, there exist two continua $K_x$ and $K_y$ such that $x \in int(K_x)$, $y \in int(K_y)$ and $K_x \cap A = \emptyset=K_y \cap A$. Let $D=K \cup K_x \cup K_y$, observe that $\{x,y\} \in int(D)$ and $D \cap A = \emptyset$. Hence, $A \in CC_1(X)$.
\end{proof}

\begin{coro}\label{semilocallyn}
    Let $X$ be a continuum, $n \in \mathbb N$ and $A \subset X$ such that $X$ is semi-locally connected at $A$. If $A \in  {NC}_n(X)$, then $A \in  {CC}_n(X)$.
\end{coro}
\begin{proof}
Assumme \(A\in NC_n(X)\). Using the same arguments as in the proof of \ref{apossemi}, we obtain that for each component $K$ of $X-A$, $K \in CC_1(A \cup K)$. Hence, by Theorem \ref{equivcomp}, $A \in CC_n(X)$.
\end{proof}

By Theorem \ref{contentions}, Proposition \ref{equivalencia}, Theorem \ref{apossemi} and Corollary \ref{semilocallyn}, we obtain the following result, which is a characterization of the elements of $CC_n(X)$.

\begin{coro}\label{ccnncnaposyndetic}
Let $X$ be a continuum and $n \in \mathbb N$. Then, $A \in CC_n(X)$ if and only if $X$ is aposyndetic respect to $A$ and $A \in NC_n(X)$.
\end{coro}

As an application of Corollary \ref{ccnncnaposyndetic} we present a continuum \(X\) in which $ NC_1(X) \cap D_0(X)= CC_1(X) \cap D_0(X)$ and \(X\) is not aposyndetic with respect to some \(A\in D_0(X)\).


\begin{example}
Let $Y$ be the harmonic fan with vertex $v$ and let $X= Y \times [0,1]/(\{v\} \times [0,1])$. Notice that \(X\) is not aposyndetic with respect to \(\{[(v,0)]\}\) and $ NC_1(X) \cap D_0(X)=\{A\in D_0(X):[(v,0)]\notin A\}= CC_1(X) \cap D_0(X)$
\end{example}


\begin{prop}
    Let $X$ be a continuum and $A \in NSC_1(X)$. If \(X\) is aposyndetic at \(p\) with respect to \(A\), for some \(p\in X-A\), then $A \in NB_1^*(X)$. 
\end{prop}
\begin{proof}
    Assume that \(X\) is aposyndetic at \(p\) with respect to \(A\), let \(C\subset X-A\) be a continuum containing \(p\) in its interior. Let \(B=\{p\}\). Take \(U\) an open set of \(X-A\).    
    Since \(X-A\) is a $1$-$Q6$ space, for $U$ and \(V=int(C)\), there exists a continuum $D\subset X-A$ such that $D \cap U \neq \emptyset\neq D \cap V$. Then \(E=C\cup D\subset X-A\) is a continuum such that \(b\in int(E)\) and \(E\cap U\neq \emptyset\). Hence, $A \in NB_1^*(X)$.
\end{proof}

In \cite[Proposition 2.2]{noncutshore} the authors proved that $NC_1(X) \cap   F_1(X) = CC_1(X) \cap   F_1(X)$ when \(X\) is aposyndetic. It is natural to ask if the converse is also true. We see in the next example that the answer is negative.

\begin{example}\label{circcantor}
    Let $\mathcal C$ be the Cantor set, $Z = \mathcal C \times \mathcal S^1$ and $X=Z/(\mathcal C \times \{q\})$, where \(q\) is a point in \(\mathcal S^1\). Notice that \(X\) is not aposyndetic at the point \([\mathcal C \times \{q\}]\) and $ NC_1(X) \cap   F_1(X) =\{\{p\}: p \in X-\{[\mathcal C \times \{q\}]\}\}=  CC_1(X) \cap   F_1(X)$.
\end{example}

For irreducible continua, the converse of \cite[Proposition 2.2]{noncutshore} is true, we give a stronger result in the following theorem.

\begin{theorem}\label{caracterizacionarco}
    Let $X$ be an irreducible continuum. If $NC_1(X)\cap F_1(X)=NWC_1(X)\cap F_1(X)$ then $X$ is an arc.
\end{theorem}
\begin{proof}
    By Corollary 2 of \cite{scc}, let $p, q \in X$ be distinct such that $\{p\}, \{q\} \in NC_1(X)$. Hence, $\{p\}, \{q\} \in NWC_1(X)$. Given that $\{p\}, \{q\} \in NWC_1(X)$, $X$ must be irreducible only between $p$ and $q$. Since $NC_1(X)\cap F_1(X)=NWC_1(X) \cap F_1(X)$ and $X$ is irreducible between $p$ and $q$, if $z \in X-\{p,q\}$, then \(\{z\}\notin NWC_1(X)\), so $\{z\}\notin NC_1(X)$. Therefore, by Section 3 of \cite{scc}, $X$ must be an arc.
\end{proof}

With Theorem \ref{caracterizacionarco}, we can provide a partial answer to Question \ref{pregunta}.

\begin{coro}\label{respuestaparcial1}
    Let $X$ be an irreducible continuum. Then $NC_1(X)=NWC_1(X)$ if and only if $X$ is an arc.
\end{coro}

The following example shows that Theorem \ref{caracterizacionarco} does not hold if we replace the hyperspace \(NWC_1(X)\) by a weaker one.
\begin{example}
    If \(X\) is a dyadic solenoid, then \(X\) is irreducible, and \(NC_1(X)\cap F_1(X)=F_1(X)=NB_1(X)\cap F_1(X)\).
\end{example}

Also, for irreducible continua, we have the following results.

\begin{theorem}\label{nwccc}
    Let $X$ be an irreducible continuum. Then $ NWC_1(X) = CC_1(X)$.
\end{theorem}

\begin{proof}
The contention $CC_1(X) \subset NWC_1(X)$ follows from Theorem \ref{contentions}.

We prove $NWC_1(X) \subset CC_1(X)$. Assume that $X$ is irreducible between $a$ and $b$ and let \(A\in NWC_1(X)\). 

\textbf{Claim 1}: If \(B\) is a component of \(A\), then \(a\in B\) or \(b\in B\). 

Proof of Claim 1. Assume \(B\) is a component of \(A\); by Theorem \ref{component}, \(B\in NWC_1(X)\). If \(a\notin B\) and \(b\notin B\), then there exists a continuum \(D\subset X-B\) such that \(a,b\in D\), which is a contradiction to the irreducibility of \(X\). Thus \(a\in B\) or \(b\in B\).

\textbf{Claim 2}: The set $A$ has at most two components. 

Proof of Claim 2. By Claim 1, each component of \(A\) contains \(a\) or contains \(b\), so \(A\) has at most two components.

\textbf{Claim 3}: If $A$ is connected, then $A \in CC_1(X)$.

Proof of Claim 3. Without loss of generality, assume $a \in A$. If $b \in A$, then $B = X \in CC_1(X)$. Assume $b \notin A$, let $x \in X - A$ and let \(E\subset X-A\) be a continuum with \(b,x\in E\). Let $V$ be an open set such that $x \in V$ and $cl(V) \cap A=\emptyset$. Let $K$ be the component of $X - V$ that contains $A$, and observe that $b \notin K$. Let $k \in K-A$; since $A \in NWC_1(X)$, there exists a continuum $D \subset X-A$ such that $b,k \in D$. As $X$ is irreducible between $a,b$, we have $D \cup K=X$, $b \notin K$ and $K \cap V=\emptyset$, so $V \subset D$, which implies $x \in int(D)$. So \(E\cup D\) is a continuum with \(b\in E\cup D\) and \(x\in int(E\cup D)\), so $A \in CC_1(X)$.

\textbf{Claim 4}: If $A$ is not connected, then $A \in CC_1(X)$.

Proof of Claim 4. Assume \(A\) is not connected. By Claim 2, \(A\) has two components $E$ and $F$. By Claim 1, without loss of generality assume $a\in E$ and $b\in F$. Let \(p\in X-A\), let $x \in X-A$ and $V$ be an open set such that $x \in V \subset cl(V) \subset X-A$. Let $K_E$ and $K_F$ be the components of $X-V$ containing $E$ and $F$ respectively. Since $K_E \cup K_F \neq X$, $K_E \cap K_F = \emptyset$. Take $e \in K_E-A$ and $f \in K_F-A$. Since $A \in NWC_1(X)$, there exists a continuum $D\subset X-A$ containing $\{e,f\}$. Therefore $D \cup K_E \cup K_F$ is a continuum containing $a$ and $b$. Since $X$ is irreducible between $a$ and $b$, $D \cup K_E \cup K_F=X$. Hence $V \subset D$. This implies $A \in CC_1(X)$.

From Claims 1,2,3,4 and 5, we conclude that $A \in NWC_1(X)$ implies $A \in CC_1(X)$. 
\end{proof}

The following example shows that the previous result does not hold if we replace $1$ by $n>1$.
\begin{example}
Let $X=\{(x,sin(\frac1x)): x \in (0,1]\} \cup \{(0,x): x \in [-1,1]\}$. Notice that \(X\) is irreducible and \(\{(0,-1)\}\in NWC_2(X)-CC_2(X)\).
\end{example}

\begin{theorem}\label{ultimoteorema}
    Let $X$ be a decomposable continuum and let $C \in NB_1(X)$. If $A,B \in C(X)-\{X\}$ are such that $A \cup B =X$ and $C \cap B = \emptyset$, then $C \in NWC_1(X)$. 
\end{theorem}
\begin{proof}
    Let $p \in X-C$, given that $C \in NB_1(X)$ and $int(B) \neq \emptyset$, there exists a continuum $K \subset X-C$, such that $p \in K$ and $K \cap B \neq \emptyset$. Therefore, $C \in NWC_1(X)$.
\end{proof}

\begin{coro}
    Let $X$ be a decomposable, irreducible continuum and let $C \in NB_1(X)$. If $A,B \in C(X)-\{X\}$ are such that $A \cup B =X$ and $C \cap B = \emptyset$, then $C \in CC_1(X)$. 
\end{coro}
\begin{proof}
Apply Theorem \ref{nwccc} and Theorem \ref{ultimoteorema}. 
\end{proof}

\begin{coro}\label{nbcc}
    Let $X$ be a decomposable continuum and irreducible between $p$ and $q$. If $\{p\} \in NB_1(X)$, then $\{p\} \in CC_1(X)$.
\end{coro}

The next example shows that in Corollary \ref{nbcc}, we can not replace the condition irreducible between $p$ and $q$ by the condition irreducible about $A \in 2^X$.
\begin{example}
Let \(X = Y \cup I\), where \(Y\) is the dyadic solenoid and \(I\) is an arc such that \(Y \cap I=\{p\}\). In this case, \(X\) is irreducible about \(I\cup \{x\}\), where \(x\in Y\) and \(p\) are in distinc composants of \(Y\). Notice \(I \in NB_1(X)\) but \(I \notin NWC_1(X)\).
\end{example}

    
By Theorem \ref{nbcc} and Lemma 3.10 from \cite{noncutshore}, we obtain the following result.
\begin{coro}
    Let $X$ be a decomposable continuum and irreducible between $p$ and $q$. If $\{p\} \in NB_1(X)$, then $X$ is locally connected at $p$.
\end{coro}

\end{document}